\documentclass{amsart}
\usepackage{verbatim}
\usepackage{ulem}
\usepackage{multirow}
\usepackage{amsfonts}
\usepackage{amsmath}
\usepackage{amssymb}
\usepackage[matrix,arrow]{xy}
\usepackage[colorlinks=true, allcolors=blue]{hyperref}
\usepackage[usenames]{color}

\newtheorem{theorem}{Theorem}
\newtheorem{proposition}[theorem]{Proposition}
\newtheorem{lemma}[theorem]{Lemma}
\newtheorem{corollary}[theorem]{Corollary}
\theoremstyle{remark}

\theoremstyle{definition}
\newtheorem{definition}[theorem]{Definition}

\title[Graded identities for pairs]{Specht property for the graded identities of the pair $(M_2(D), sl_2(D))$}

\author[Códamo]{Ramon Códamo}
\thanks{R. Códamo was financed in part by the Coordena\c c\~ao de
Aperfei\c{c}oamento de Pessoal de N\'\i vel Superior - Brasil (CAPES) - Finance Code 001}
\address{Department of Mathematics, UNICAMP, 13083-859 Campinas, SP,  Brazil}
\email{ramoncodamo@gmail.com}

\author[Koshlukov]{Plamen Koshlukov}
\thanks{P. Koshlukov was partially supported by FAPESP grant No.~2018/23690-6 and by CNPq grant No.~302238/2019-0}
\address{Department of Mathematics, UNICAMP, 13083-859 Campinas, SP,  Brazil}
\email{plamen@ime.unicamp.br}

\keywords{Specht property; graded polynomial identities; weak identities}

\subjclass{16R10, 16R50, 17B01}
\begin{document}

\maketitle

\begin{abstract}
Let $D$ be a Noetherian infinite integral domain, denote by $M_2(D)$ and by $sl_2(D)$ the $2\times 2$ matrix algebra and the Lie algebra of the traceless matrices in $M_2(D)$, respectively. In this paper we study the natural grading by the cyclic group $\mathbb{Z}_2$ of order 2 on $M_2(D)$ and on $sl_2(D)$. We describe a finite basis of the graded polynomial identities for the pair $(M_2(D), sl_2(D))$. Moreover we prove that the ideal of the graded identities for this pair satisfies the Specht property, that is every ideal of graded identities of pairs (associative algebra, Lie algebra), satisfying the graded identities for $(M_2(D), sl_2(D))$, is finitely generated. The polynomial identities for $M_2(D)$ are known if $D$ is any  field of characteristic different from 2. The identities for the Lie algebra $sl_2(D)$ are known when $D$ is an infinite field. The identities for the pair we consider were first described by Razmyslov when $D$ is a field of characteristic 0, and afterwards by the second author when $D$ is an infinite field. The graded identities for the pair $(M_2(D), gl_2(D))$ were also described, by Krasilnikov and the second author.

In order to obtain these results we use certain graded analogues of the generic matrices, and also techniques developed by G. Higman concerning partially well ordered sets.
\end{abstract}

\section*{Introduction}
The bulk of modern PI theory is devoted to the study of polynomial identities satisfied by algebras over a field $F$, and most of the research assumes that the field is of characteristic 0. The description of the identities satisfied by an important algebra is one of the classical problems in PI theory. But it also turned out to be among the hardest ones. The identities satisfied by the infinite dimensional Grassmann algebra $E$ were described by Latyshev \cite{latyshev} and, in great detail,  by Krakowski and Regev \cite{krreg}. The identities satisfied by $M_2(F)$, the full matrix algebra of order 2, were found by Razmyslov \cite{razmm2}, in characteristic 0, and by the second author in \cite{pkm2} over infinite fields of characteristic different from 2. It is a long-standing open problem to determine whether the identities of $M_2(F)$ admit a finite generating set in characteristic 2. The polynomial identities satisfied by the upper triangular matrices are well known in any characteristic. The identities of the tensor square of the Grassmann algebra, $E\otimes E$, in characteristic 0, were obtained by Popov \cite{popov}. And these comprise,  more or less, the list of algebras where the identities are known. In 1950 Specht posed a problem that shaped a good part of the development of PI theory: If $F$ is a field of characteristic 0, and $A$ is an associative $F$-algebra satisfying polynomial identities, is the ideal of identities $T(A)$ finitely generated? It was solved in the affirmative by Kemer in 1984--1986, see \cite{kemer}. Clearly one can ask a similar question for other classes of algebras and/or other fields. In the case of Lie algebras over infinite fields of characteristic 2, Vaughan-Lee gave a counterexample \cite{vlee}; later on Drensky constructed examples (even finite-dimensional) over any infinite field of positive characteristic \cite{drenskyspecht}. Counterexamples for associative algebras over an infinite field of positive characteristic were discovered much later \cite{Be, Gr, VS}. In characteristic 0, Kemer's result holds for Lie \cite{ilt_lie}, alternative \cite{ilt_alt}, Jordan \cite{vz} algebras, but under some additional restrictions. If $A$ is an associative algebra then under the product (commutator) $[a,b]=ab-ba$ it becomes a Lie algebra denoted by $A^{(-)}$. By the well known PBW theorem, every Lie algebra $L$ is a subalgebra of some $A^{(-)}$. Analogously if one substitutes the associative product by the symmetric (or Jordan) product $a\circ b=ab+ba$ then it becomes a Jordan algebra denoted by $A^{(+)}$. The Jordan algebras $A^{(+)}$ and their subalgebras are called special, otherwise they are exceptional. Contrary to the Lie case, there exist exceptional Jordan algebras. In \cite{razmm2} Razmyslov introduced the so called weak polynomial identities, these are called sometimes identities of representations of Lie algebras. Let $L$ be a Lie algebra and suppose $L\subseteq A^{(-)}$ for some associative algebra $A$. Suppose further that $L$ generates $A$ as an algebra. A polynomial in the free associative algebra $F\langle X\rangle$ is a weak identity for the pair $(A,L)$ if it vanishes (as an element of $A$) under every substitution of its variables by elements of $L$. For example, if $x\in sl_2(F)$ is a traceless matrix then its square is a scalar matrix, hence one has the weak identity $[x^2,y]$ for the pair $M_2(F), sl_2(F))$. Razmyslov \cite{razmm2} described the weak identities for the pair $(M_2(F),sl_2(F))$ in characteristic 0, and this was crucial in his description of finite bases of the identities for both $M_2(F)$ and $sl_2(F)$. Later on the identities of $sl_2(F)$ were described when $F$ is an infinite field of characteristic $p>2$ by Vasilovsky \cite{vasal}; the weak identities of $(M_2(F),sl_2(F))$ in \cite{kjawpi}. It turns out that the weak identities for the pair $(M_2(F),sl_2(F))$ are determined by the single identity above. Weak identities have been used in the study of identities in nonassociative algebras to great effect, see for example \cite{razmm2, vasal, pkm2}. All this comes to say that studying weak polynomial identities is an important part of PI theory. On one hand they tend to be "easier" to investigate than the ordinary ones, on the other hand they still provide valuable information about the ordinary (associative, Lie or whatever) identities in the corresponding algebras. 

Introducing an additional structure on the algebras under consideration is a powerful tool in studying various aspects of the algebras. More specifically, for the PI case, one may study algebras with involution, or with trace, and "incorporate" the additional operation into the identities. In this way very general results were obtained, and important methods were developed, or adapted to PI algebras. Group gradings on algebras first appeared long ago: the commutative polynomial algebra in one or several variables is naturally graded by the integers $\mathbb{Z}$, by the degree. These gradings are of significant importance in Commutative Algebra, and in practically every branch of Mathematics. Group gradings on Lie algebras arose more than a century ago, in relation to the root systems of semisimple Lie algebras. Later on gradings by the group of order 2, $\mathbb{Z}_2$, turned out extremely useful. These led to the extremely important notion of a Lie (or Jordan) superalgebra. Recall here that in the case of associative algebras, $A$ is $\mathbb{Z}_2$-graded whenever $A=A^{(0)}\oplus A^{(1)}$, a direct sum of vector subspaces such that $A^{(i)}A^{(j)}\subseteq A^{(i+j)}$ where $i+j$ is considered modulo 2. For associative algebras, there is no difference between $\mathbb{Z}_2$-graded algebras and superalgebras. On the other hand, while a Lie (or Jordan) superalgebra is indeed $\mathbb{Z}_2$-graded, it is seldom a Lie (or Jordan) algebra. We do not enter here into details as we will not need the notion of superalgebras in this paper.

Graded identities have been studied quite extensively, and the results obtained are far more general than their non-graded counterparts. Thus for example the graded identities for the matrix algebras $M_n(F)$, when $F$ is infinite, were described in \cite{vaspams} and in \cite{ssa}, where the grading is the natural one by the cyclic group of order $n$. Analogues of Kemer's theory for associative algebras graded by finite groups were also obtained \cite{AB}, \cite{IS}. Here we recall that the relation between the graded and the ordinary identities satisfied by a given algebra is not very strong, and one can hardly describe the ordinary identities by just knowing the graded ones. On the other hand, if two graded algebras satisfy the same graded identities then they satisfy the same ordinary identities. This latter fact is frequently used in PI theory.

Let $A$ be an associative algebra and let $F\langle X\rangle$ be the free associative algebra freely generated over $F$ by the infinite set $X=\{x_1,x_2,\ldots\}$. The set of all polynomials $f(x_1,\ldots,x_n)\in F\langle X\rangle$ such that $f(a_1,\ldots,a_n)=0$ in $A$ for every $a_i\in A$ is the T-ideal of $A$ (the ideal of polynomial identities of $A$). Clearly $T(A)$ is closed under endomorphisms of $F\langle X\rangle$, and every ideal in $F\langle X\rangle$ with this property is the ideal of identities of an algebra. If $S\subseteq T(A)$ and $S$ generates $T(A)$ as a T-ideal then $S$ is called a basis of $T(A)$. A T-ideal $I$ satisfies (or has) the Specht property if it admits a finite basis, and every T-ideal containing $I$ also admits a finite basis. Kemer's theorem mentioned above states that every T-ideal in the free associative algebra, over a field of characteristic 0, satisfies the Specht property. One defines the above notions for graded identities in an analogous way. If $G$ is a finite group then the results from \cite{AB} and \cite{IS} read that if $G$ is a finite group then every $G$-graded T-ideal in characteristic 0 has the Specht property. 

In this paper we study the $\mathbb{Z}_2$-graded weak identities for the pair $(M_2(D), sl_2(D))$ where $D$ is an infinite (unital) domain. We determine explicitly a finite basis of the graded weak identities for it. Furthermore we establish the Specht property for the ideal of weak graded identities for the above pair. We impose an obvious restriction on $D$, namely we require $D$ to be Noetherian; in this case one can consider even a more general situation where $D$ is a Noetherian commutative unital ring. In order to obtain a finite basis for the graded weak identities of $(M_2(D), sl_2(D))$ we use an adequate model for the relatively free graded pair, namely generic matrices, and their evaluations. For the Specht property we use methods and techniques that go back to the paper by G. Higman \cite{higman}, concerning partially well ordered sets (PWOS). We recall that PWOS have been used extensively in PI theory for establishing the Specht property of certain T-ideals.

\section{Preliminaries}

Let $G=\mathbb{Z}_2=\{0,1\}$ (we use the additive notation), and let $Y=\{y_1,y_2,\ldots\}$, $Z=\{z_1,z_2,\ldots\}$ be disjoint infinite countable sets. Fix a field $F$, and form the free unital associative algebra $F\langle X\rangle$ where $X=Y\cup Z$. It is 2-graded in the following way. We declare the variables from $Y$ even (of degree 0), and the ones from $Z$ odd (of degree 1), and extend this to the monomials (that is to the basis of the vector space $F\langle X\rangle$). Hence a monomial is of degree 0 whenever it contains an even number of variables from the set $Z$ (counting their multiplicities). If $A=A^{(0)}\oplus A^{(1)}$ is a $G$-graded algebra, a polynomial $f(y_1,\ldots, y_k, z_1,\ldots, z_m)$ is a graded identity for $A$ if it vanishes on $A$ under every substitution of the $y_i$ by elements from $A^{(0)}$ and the $z_i$ by elements from $A^{(1)}$. In other words $f\in \ker\varphi$ for every algebra homomorphism $\varphi\colon F\langle X\rangle\to A$ that respects the grading. Such homomorphisms are called graded ones. 

Let $L\langle X\rangle \subseteq F\langle X\rangle$ be the Lie subalgebra of $F\langle X\rangle^{(-)}$ generated by the set $X$. By the well known Witt theorem $L\langle X\rangle$ is the free Lie algebra; it inherits the grading from $F\langle X\rangle$; moreover it is the free graded Lie algebra. Clearly the pair $(F\langle X\rangle, L\langle X\rangle)$ is free in the class of all pairs, and it is freely generated by $X$.  An ideal $I$ in $F\langle X\rangle$ is a weak T-ideal whenever it is closed under all graded endomorphisms of the pair $(F\langle X\rangle, L\langle X\rangle)$. If $S\subseteq F\langle X\rangle$ then the intersection $J$ of all weak T-ideals that contain $S$ is the weak T-ideal generated by $S$. Clearly $J$ coincides with the homogeneous (in the grading) ideal generated by the polynomials
\[
\{ f(p_1, \ldots, p_k, q_1,\ldots, q_m)\mid f(y_1,\ldots,y_k, z_1,\ldots, z_m)\in S, p_i\in L\langle X\rangle^{(0)}, q_j\in  L\langle X\rangle^{(1)}\}.
\]
Let $(A,L)$ be a graded pair, denote by $Id(A,L)$ the weak T-ideal of the graded identities for it. If $I$ is a weak T-ideal in $F\langle X\rangle$ then $I=Id(A,L)$ where $A=F\langle X\rangle/I$ and $L=L\langle X\rangle+I/I$. If $S\subseteq F\langle X\rangle$ is a set of homogeneous polynomials we denote by $\mathcal{V}=\mathcal{V}(S)$ the class of all graded pairs that satisfy the weak identities from $S$. This class is the variety of graded pairs determined by $S$.  As in the case of ordinary polynomial identities one can prove that if $I=Id(A,L)$ then the pair $(F\langle X\rangle/I, L\langle X\rangle+I/I)$ is relatively free in the variety determined by $I$. 

Let $S\subseteq F\langle X\rangle$ be a nonempty set of homogeneous polynomials. An element $f\in F\langle X\rangle$ is a consequence of the identities of the set $S$ if $f$ lies in the least ideal of weak graded identities that contains $S$. Sometimes we shall say that $f$ follow from the identities in $S$. Suppose $I$ is an ideal of weak graded identities and $S\subseteq I$. The set $S$ is called a basis of $I$ whenever $S$ generates $I$ as an ideal of weak graded identities. We draw the readers' attention that we do not require minimality of $S$. If $f$ and $g$ are two homogeneous elements of $F\langle X\rangle$ then they are equivalent as weak graded identities whenever $f$ follows from $g$ and $g$ follows from $f$. 

The matrix algebra $A=M_2(F)$ is $\mathbb{Z}_2$-graded in a natural way where the diagonal matrices form $A^{(0)}$ and the off-diagonal ones form $A^{(1)}$. Then $sl_2(F)\subseteq M_2(F)$ inherits this grading. The graded  identities of $M_2(F)$ were first described by Di Vincenzo in \cite{divinc} in characteristic 0, and later on in \cite{pkssa} this result was extended to any infinite field. It was proved in \cite{divinc, pkssa} that they follow from the graded identities $y_1y_2-y_2y_1$ and $z_1z_2z_3-z_3z_2z_1$. As commented above the graded identities for the Lie algebra $sl_2(F)$ were obtained in \cite{pk_sl2}, they are consequences of the single identity $[y_1,y_2]$. The graded identities for the pair $(M_2(F), gl_2(F))$ were described in \cite{kk}. In the latter paper $F$ was assumed an infinite integral domain.

Clearly all the above definitions and notions hold when one substitutes the field $F$ by an infinite unital domain $D$. Hence from now on we assume that $F=D$ is such a domain. Let $A=D[\alpha_i,\beta_i,\gamma_i\mid i\ge 1]$ be the unital commutative polynomial algebra in the variables $\alpha_i$, $\beta_i$, $\gamma_i$. We form the  associative algebra $M_2(A)$ and denote by $R$ its unital (associative) subalgebra generated by the matrices
\begin{equation}
\label{genmatr}
Y_i=\begin{pmatrix} \alpha_i& 0\\ 0& -\alpha_i\end{pmatrix}, \qquad Z_i=\begin{pmatrix} 0&\beta_i\\ \gamma_i&0\end{pmatrix}, \qquad i=1, 2, \ldots,
\end{equation}
and by $S$ the Lie subalgebra of $M_2(A)^{(-)}$  generated by the $Y_i$ and $Z_i$. Then $R=R_0\oplus R_1$ is naturally $\mathbb{Z}_2$-graded by assuming the $Y_i$ even and the $Z_i$ odd elements, and analogously for $S=S_0\oplus S_1$. Since $S_i=R_i\cap S$ we obtain that $(R,S)$ is a graded pair.

The following proposition is well known in the case of associative algebras, as well as in the case of (non-graded) pairs. Its proof repeats that of the associative case, and that is why we omit it.

\begin{proposition}
The pair $(R,S)$ is relatively free in the variety of graded pairs determined by $Id(M_2(D), sl_2(D))$.
\end{proposition}

We recall some basic facts concerning partially well ordered sets. Let $(A, \le)$ be a partially ordered set. If $B\subseteq A$ then the closure $cl(B)$ in $A$ is the set of all $a\in A$ such that for some $b\in B$ one has $b\le a$. The subset $B$ is closed whenever $B=cl(B)$. The following statement can be found in Higman's paper \cite{higman}.

\begin{proposition}
\label{higman}
Let $A$ be a partially ordered set. The following conditions are equivalent.
\begin{enumerate}
\item
Every infinite sequence $a_1$, $a_2$, \dots, of elements of $A$, has an infinite subsequence $a_{i_1}\le a_{i_2}\le \cdots$ where $i_1<i_2<\cdots$;
\item
For every infinite sequence $a_1$, $a_2$, \dots, of elements of $A$ there exist indices $i<j$ such that $a_i\le a_j$;
\item
Every ascending chain of closed subsets of $A$ stabilizes;
\item
Every nonempty closed subset of $A$ is the closure of a finite subset of $A$;
\item
There are neither infinite decreasing sequences of elements of $A$ nor infinite sequences of pairwise incomparable elements of $A$;
\item
Every nonempty subset of $A$ has finitely many minimal elements.
\end{enumerate}
\end{proposition}
A partially ordered set $A$ satisfying one (and hence all) of the conditions of the proposition is called partially well ordered set (PWOS). 

The following theorem can also be found in \cite{higman}.

\begin{theorem}
\label{higman_pwos}
\begin{enumerate}
\item
Let $(A,\le)$ be a partially ordered set, and form the set $V(A)$ of all finite sequences of elements of $A$. Define a partial order $\le_s$ on $V(A)$ as follows: if $u$ and $v\in V(A)$ then $u\le_s v$ whenever there is a subsequence of $v$ that majorizes element-wise the sequence $u$. Then, if $(A,\le)$ is a PWOS so is $(V(A),\le_s)$.
\item
If $(A_1,\le_1)$ and $(A_2,\le_2)$ are two PWOS then their direct product $A=A_1\times A_2$ is a PWOS with respect to the partial order induced by $\le_1$ and $\le_2$: $(a_1, a_2)\le (a_1',a_2')$ whenever $a_1\le_1 a_1'$ and $a_2\le_2 a_2'$.
\end{enumerate}
\end{theorem}

\section{A basis for the graded identities}

In this section we fix an infinite unital domain $D$. 
Recall that $X=Y\cup Z$ where the variables from $Y$ are even, and those of $Z$ are odd. 
\begin{lemma}
\label{basicids}
The following graded identities hold for the graded pair $(M_2(D), sl_2(D))$:
\[
[y_1,y_2], \qquad z_1z_2z_3-z_3z_2z_1, \qquad yz+zy.
\]
\end{lemma}
\begin{proof}
Omitted since it consists of a direct (and easy) verification.
\end{proof}

We denote by $I$ the ideal of weak graded identities generated by the identities from the above lemma. Let $B$ be the set of all monomials of the form
\begin{equation}
\label{basisrfa}
y_{a_1} y_{a_2}\cdots y_{a_k}, \qquad y_{a_1} y_{a_2}\cdots y_{a_k} z_{i_1}z_{j_1} \cdots z_{i_n}\widehat{z_{j_n}},
\end{equation}
where $k\ge 0$, $n\ge 0$, $a_1\le a_2\le \cdots\le a_k$, $i_1\le i_2\le \cdots\le i_n$, $j_1\le j_2\le \cdots\le j_n$, and the ``hat'' over a variable (or an expression) means that it may be missing. In case $k=n=0$ the resulting element is 1.

\begin{lemma}
\label{span}
The $D$-module $D\langle X\rangle/I$ is spanned by the images of the elements from Eq.~(\ref{basisrfa}).
\end{lemma}
\begin{proof}
By using several times the third identity from Lemma~\ref{basicids} we can put the variables $y$ at the beginning of each monomial in $D\langle X\rangle/I$. Then we use the first identity to order the variables $y$. Then the second identity gives us the correct order for the variables $z$.
\end{proof}
\begin{proposition}
The images of the elements from Eq.~(\ref{basisrfa}) form a basis of the $D$-module $D\langle X\rangle/I$.
\end{proposition}
With certain abuse of notation we shall denote by the same letter $B$ the set of these images. 

\begin{proof}
Since $D$ is infinite we can consider the multihomogeneous components, in the usual sense, of $D\langle X\rangle/I$. 
We use an argument similar to that of \cite[Proposition 5]{pkssa}. If $m_1=y_{i_1}y_{i_2}\cdots y_{i_k}$ is a monomial, we evaluate it on the generic matrices $Y_{i_1}$, $Y_{i_2}$, \dots, $Y_{i_k}$:
\[
m_1(Y_{i_1}, Y_{i_2}, \ldots, Y_{i_k}) = \begin{pmatrix} \alpha_{i_1} \alpha_{i_2}\cdots \alpha_{i_k} & 0 \\ 0&(-1)^k \alpha_{i_1} \alpha_{i_2}\cdots \alpha_{i_k}\end{pmatrix}.
\]
This means that the monomials depending on variables $y$ only are linearly independent. Pay attention that in the evaluation of these monomials the variables $\beta$ and $\gamma$ do not appear in the entries of the corresponding matrix. 

Now suppose $m_2=z_{i_1}z_{j_1}\cdots z_{i_n}\widehat{z_{j_n}}$ where $k=0$, that is no variables $y$ appear in $m_2$. Once again evaluating on the generic matrices $Z$ we obtain
\[
m_2(Z_{i_1}, Z_{j_1}, \ldots, Z_{i_n}, \widehat{Z_{j_n}}) = 
\begin{cases} (\beta_{i_1}\gamma_{j_1} \cdots \beta_{i_n}\gamma_{i_n})e_{11}+(\gamma_{i_1}\beta_{j_1}\cdots \gamma_{i_n}\beta_{j_n}) e_{22}, \\
 (\beta_{i_1}\gamma_{j_1} \cdots \beta_{i_n})e_{12}+(\gamma_{i_1}\beta_{j_1}\cdots \gamma_{i_n}) e_{21},
\end{cases}
\]
depending on whether $z_{j_n}$ appears, or not,  in $m_2$. Clearly the above evaluation recovers uniquely the monomial $m_2$, and in this way we have that the monomials are linearly independent. Pay attention that no variable $\alpha$ appears in the entries of these matrices.

Now consider the general case, where $m$ is a monomial with variables $y$ and $z$. Then the variables $\alpha$ in the entries of the corresponding matrix identify uniquely the graded variables $y$, and as above we can recover the order of the graded variables $z$, thus proving the statement of the proposition. 
\end{proof}

\begin{corollary}
Let $D$ be an infinite unital domain, then the graded identities from Lemma~\ref{basicids} form a basis of the graded identities for the pair $(M_2(D), sl_2(D))$.
\end{corollary}

\section{The Specht property}

In this section we require $D$ to be an infinite unital Noetherian domain. We recall that $I$ is the ideal of weak graded identities generated by the polynomials from Lemma~\ref{basicids}. We split the basis $B$ of the free $D$-module $D\langle X\rangle/I$ into two subsets:
\[
M_1=\{y_{a_1} y_{a_1}\cdots y_{a_k}\}, \quad M_2=\{y_{a_1} y_{a_1}\cdots y_{a_k} z_{c_1}z_{d_1}\cdots z_{c_m}\widehat{z_{d_m}}\},
\]
where $a_1\le a_2\le\cdots\le a_k$, $k\ge 0$, $c_1\le c_2\le\cdots\le c_m$, $d_1\le d_2\le\cdots\le d_m$, $m\ge 1$. 

The set $\mathbb{N}_0=\mathbb{N}\cup \{0\}$ has a natural order $\le$. We form the set $Q_n=\mathbb{N}_0^n$ of the $n$-tuples of nonnegative integers, and define on $Q_n$ a partial order, denoted by the same symbol $\le$, as follows: $(a_1, \ldots, a_n)\le (b_1,\ldots,b_n)$ whenever $a_i\le b_i$ for every $i$. Since $\mathbb{N}_0$ is a PWOS then Theorem~\ref{higman_pwos} implies that $(Q_n, \le)$ is also a PWOS. Once again Theorem~\ref{higman_pwos} yields that the set of the finite sequences $V(Q_n)$ with the order $\le_s$ is a PWOS.

We consider the set $V(Q_1)$, and we define an order $\le$ on it as follows: if $\{u_i\}$ and $\{v_i\}$ are two finite sequences of nonnegative integers, we put $\{u_i\}\le \{v_i\}$ whenever either $\{u_i\} = \{v_i\}$ or there exists $k>0$ such that $u_k<v_k$ and $u_j=v_j$ for every $j>k$. (If necessary we complete the shorter of the two sequences with zeros on the right-hand side.) Then the set $(V(Q_1), \le)$ is linearly well ordered (that is the order is linear and every nonempty subset has the least element). If we rewrite the set of the monomials $M_1$ as $M_1=\{y_1^{u_1}y_2^{u_2}\cdots y_n^{u_n}\mid u_i\ge 0\}$ then clearly $M_1$ is in a 1--1 correspondence with the set $V(Q_1)$, and the order on the latter is a slight modification of the right lexicographic order on $M_1$. 

There is another natural 1--1 correspondence, this time of the sets $V(Q_n)$ and $V(Q_1)^n$: given an $n$-tuple of finite sequences $u=(u_i^{(1)}, \ldots, u_i^{(n)})\in V(Q_n)$ we associate to it $v\in V(Q_1)^n$ as follows: $u=(u_i^{(1)}, \ldots, u_i^{(n)}) \leftrightarrow v= (\{u_i^{(1)}\}, \ldots, \{u_i^{(n)}\})$. We shall use this bijection to identify the two sets when necessary.

In what follows we use some ideas from \cite{evander}. Define a function $\partial\colon V(Q_n)\to \mathbb{N}_0$ by $\partial(u) = \partial(\{u_i^{(1)}\}, \ldots, \{u_i^{(n)}\}) = \sum_{i,j} u_i^{(j)}$. Here we identify $V(Q_n)$ with $V(Q_1)^n$; the sum on the right-hand side is finite. Now take $u=(\{u_i^{(1)}\},  \{u_i^{(2)}\})$, $v=(\{v_i^{(1)}\}, \{v_i^{(2)}\})\in V(Q_2)$. We define a partial order $\le$ on $V(Q_2)$ by means of $u<v$ whenever either $\partial(u)<\partial(v)$, or $\partial(u)=\partial(v)$ but $\{u_i^{(2)}\}<\{v_i^{(2)}\}$, or else $\partial(u)=\partial(v)$, $\{u_i^{(2)}\}=\{v_i^{(2)}\}$, but $\{u_i^{(1)}\}<\{v_i^{(1)}\}$. Here the latter order is that on $V(Q_1)$. An easy verification shows that $\le$ turns $V(Q_2)$ into a linearly well ordered set. We use this order to define one on $V(Q_n)$. 

Suppose $u=(\{u_i^{(1)}\}, \ldots, \{u_i^{(n)}\})$, $v=(\{v_i^{(1)}\}, \ldots, \{v_i^{(n)}\})\in V(Q_n)$, $n\ge 3$. Then define $u<v$ whenever either $(\{u_i^{(n-1)}\}, \{u_i^{(n)}\})< (\{v_i^{(n-1)}\}, \{v_i^{(n)}\})$ as elements of $V(Q_2)$, or $(\{u_i^{(n-1)}\}, \{u_i^{(n)}\})= (\{v_i^{(n-1)}\}, \{v_i^{(n)}\})$ but there exists $t_0$ such that $\{u_i^{(t_0)}\} < \{v_i^{(t_0)}\}$ and $\{u_i^{(t)}\} = \{v_i^{(t)}\}$ for every $t>t_0$. Then this gives us that $V(Q_n)$ is linearly well ordered set. 

We form the sets $V_1=V(Q_1)$, $V_2=\{v\in V(Q_3)\mid \partial(v_i^{(2)})>0, \partial(v_i^{(2)})- \partial(v_i^{(3)})=0 \text{ or } 1\}$, and finally $V=V_1\cup V_2$. The sets $V_1$ and $V_2$ inherit the corresponding orders from $V(Q_1)$ and from $V(Q_3)$, respectively. Thus we define an order $\le$ on $V$ by using the ones on $V_1$ and on $V_2$, and by declaring $V_1<V_2$.

\begin{lemma}
\label{Vislwo}
The set $(V,\le)$ is linearly well ordered.
\end{lemma}

\begin{proof}
Clearly $\le$ is a linear order. If $\emptyset\ne W\subset V$ and $W\cap V_1\ne\emptyset$ then this intersection has its least element which is the least element in $W$. Otherwise, if $W\cap V_1=\emptyset$ then $W\subseteq V_2$ and $W$ has its least element.
\end{proof}

We define another order $\le'$ on $V$ as follows. If $u$, $v\in V$ then $u\le' v$ if $u$, $v\in V_i$ for some $i=1$, 2, and $u\le_s v$ in the order of $V_i$ (see Theorem \ref{higman_pwos} for the definition of $\le_s$).

\begin{lemma}
\label{Vispwos}
The set $(V, \le')$ is a PWOS.
\end{lemma}

\begin{proof}
Clearly $\le'$ is a partial order. If $v=\{v_i\}$ is an infinite sequence of elements of $V$ then it has an infinite subsequence of elements of the same $V_i$. But every infinite sequence of elements of $V_1$ or of $V_2$ has an infinite increasing (with possible repetitions) subsequence. Then we apply Proposition~\ref{higman}.
\end{proof}

For every monotone injection $\varphi\colon\mathbb{N}_0\to\mathbb{N}_0$, we define, abusing with the notation, a function $\varphi\colon V(Q_n)\to V(Q_n)$ as follows: $\varphi(\{u_i\})=\{v_i\}$ where $v_j=\varphi(u_i)$ whenever $j=\varphi(i)$ for some $i$, and $v_j=0$ otherwise. We denote by $\Phi_1$ the set of all functions $\varphi_1\colon V(Q_n)\to V(Q_n)$ such that $\varphi_1=\varphi$ in case $n=1$ or 2, and $\varphi_1(u,v) = (\varphi(u),v)$ if $n>2$, where $u\in V(Q_{n-2})$, $v\in V(Q_2)$. Furthermore we define $\Phi_2$ as the set of all $\varphi_2\colon V(Q_n)\to V(Q_n)$ such that $\varphi_2$ is the identity on $V(Q_1)$ and on $V(Q_2)$, and in case $n>2$, put $\varphi_2(u,v) = (u, \varphi(v))$ for $u\in V(Q_{n-2})$, $v\in V(Q_2)$.

For every monotone injection $\varphi\colon \mathbb{N}_0\to\mathbb{N}_0$ we define an endomorphism $\widetilde{\varphi}$ of the algebra $D\langle X\rangle/I$ as $\widetilde{\varphi}(x_i+I) = x_{\varphi(i)}+I$, for each $x_i\in X=Y\cup Z$. Let $\widetilde{\Phi}$ be the set of these endomorphisms. We define further $\widetilde{\Phi}_1$ as the set of the endomorphisms of $D\langle X\rangle/I$ mapping the $y_i+I$ into $y_{\varphi(i)}+I$ and fixing the $z_i+I$. Moreover $\widetilde{\Phi}_2$ is the set of the endomorphisms of $D\langle X\rangle/I$ that are the identity on $y_i+I$, and that map $z_i+I$ to $z_{\varphi(i)}+I$ for every $i$. 

The set of monomials $B$ was defined in Eq.~(\ref{basisrfa}), and we put $B_1$ to be its subset formed by the $y_{a_1}y_{a_2}\cdots y_{a_k}$, $k\ge 0$, and by $B_2=B\setminus B_1$. In order to keep the notation simpler we use the same letters $B$, $B_1$, $B_2$ for these monomials and for their images under the canonical projection onto $D\langle X\rangle/I$. We define the function $\xi\colon B\to V$ by means of
\begin{align}
\label{xi}
\xi(y_{a_1}y_{a_2}\cdots y_{a_k}) &= \{u_i^{(1)}\},  \\
\xi(y_{a_1}y_{a_2}\cdots y_{a_k}z_{c_1}z_{d_1}\cdots z_{c_m}\widehat{z_{d_m}})&= (\{u_i^{(1)}\}, \{u_i^{(2)}\}, \{u_i^{(3)}\}).\nonumber
\end{align}
Here $u_i^{(1)}= \sum_{a_j=i} 1$, $u_i^{(2)} = \sum_{c_j=i}1$, $u_i^{(3)} = \sum_{d_j=i}1$. Clearly $\xi$ is a bijection of $B$ and $V$; we use this in order to transfer the orders $\le$ and $\le'$ from $V$ to $B$.

\begin{lemma}
\label{comp}
For every monomial $M\in B$ one has $\xi(\widetilde{\varphi}(M)) = \varphi(\xi(M))$ and $\xi(\widetilde{\varphi}_i(M)) = \varphi_i(\xi(M))$ for each $\widetilde{\varphi}\in \widetilde{\Phi}$, $\widetilde{\varphi}_i \in \widetilde{\Phi}_i$, $i=1$, 2.
\end{lemma}

\begin{proof}
We have $\widetilde{\varphi}= \widetilde{\varphi}_1\circ \widetilde{\varphi}_2 = \widetilde{\varphi}_2\circ \widetilde{\varphi}_1$, hence it suffices to prove that $\xi(\widetilde{\varphi}_i(M)) = \varphi_i(\xi(M))$. Suppose first that $M= y_{a_1}\cdots y_{a_k} = y_1^{u_1^{(1)}} \cdots y_n^{u_n^{(1)}}$, so that $u_i^{(1)}= \sum_{a_j=i} 1$. Then $\varphi_1(\xi(y_1^{u_1^{(1)}} \cdots y_n^{u_n^{(1)}})) = \varphi_1(\{u_i^{(1)}\}) = \varphi(\{u_i^{(1)}\}) = \{v_i\}$. Here we have $v_j=u_i^{(1)}$ whenever $j=\varphi(i)$ for some $i$, and $v_j=0$ otherwise. On the other hand
\[
\xi(\widetilde{\varphi}_1(M)) = \xi(\widetilde{\varphi}_1(y_1^{u_1^{(1)}} \cdots y_n^{u_n^{(1)}})) = \xi(y_{\varphi(1)}^{u_1^{(1)}} \cdots y_{\varphi(n)}^{u_n^{(1)}}) = \{v_i\},
\]
where $v_j=u_i^{(1)}$ in case $j=\varphi(i)$ for some $i$, and $v_j=0$ otherwise. Thus for $M\in B_1$ the statement holds. 

Suppose $M=y_{a_1}y_{a_2}\cdots y_{a_k}z_{c_1}z_{d_1}\cdots z_{c_m}\widehat{z_{d_m}} = y_1^{u_1^{(1)}} \cdots y_n^{u_n^{(1)}} z_{c_1}z_{d_1}\cdots z_{c_m}\widehat{z_{d_m}}$ lies in $B_2$. Then
\[
\varphi_1(\xi(M)) = 
  	\varphi_1(\{u_i^{(1)},u_i^{(2)},u_i^{(3)}\}) = 
  	(\varphi(\{u_i^{(1)}\}),\{(u_i^{(2)},u_i^{(3)})\}) = 
  	\{(v_i,u_i^{(2)},u_i^{(3)})\},
\]
where $v_j=u_i^{(1)}$ if $j=\varphi(i)$ for some $i$, and $v_j=0$ otherwise,  $u_i^{(2)}=\sum_{c_j=i}1$, and $u_i^{(3)}=\sum_{d_j=i}1$. On the other hand 
\[
\xi(\widetilde{\varphi}_1(M)) =  
  	\xi(\widetilde{\varphi}_1(y_1^{u_1^{(1)}}\cdots y_n^{u_n^{(1)}}z_{c_1}z_{d_1}\cdots z_{c_m}\widehat{z_{d_m}} )) = \{(v_i,u_i^{(2)},u_i^{(3)})\},
\]
and we are done with $\varphi_1$. 

Now we consider $\varphi_2$. Let $M = y_1^{u_1^{(1)}} \cdots y_n^{u_n^{(1)}}\in B_1$, then
\[
\xi(\widetilde{\varphi}_2(M)) = \xi(\widetilde{\varphi}_2(y_1^{u_1^{(1)}}\cdots y_n^{u_n^{(1)}} ))= 
  	\xi(y_1^{u_1^{(1)}}\cdots y_n^{u_n^{(1)}} )= \{u_i^{(1)}\}.
\]
We also have
\[
\varphi_2(\xi(M)) = \varphi_2(\xi(y_1^{u_1^{(1)}}\cdots y_n^{u_n^{(1)}} )) = 
  	\varphi_2(\{u_i^{(1)}\}) = \{u_i^{(1)}\},
\]
and this case is done. Let $M =y_1^{u_1^{(1)}}\cdots y_n^{u_n^{(1)}}z_{c_1}z_{d_1}\cdots z_{c_m}\widehat{z_{d_m}} \in B_2$, then
\begin{align*}
\xi(\widetilde{\varphi}_2(M)) &= 
\xi(\widetilde{\varphi}_2(y_1^{u_1^{(1)}}\cdots y_n^{u_n^{(1)}}z_{c_1}z_{d_1}\cdots z_{c_m}\widehat{z_{d_m}} )) \\
&= 
\xi(y_1^{u_1^{(1)}}\cdots y_n^{u_n^{(1)}}z_{\phi(c_1)}z_{\phi(d_1)}\cdots z_{\varphi(c_m)}\widehat{z_{\varphi(d_m)}} )) = 
\{u_i^{(1)},u_i^{(2)}, u_i^{(3)}\},
\end{align*}
where the $\{u_i^{(k)}\}$ are defined as above. To conclude, we compute
\[
\varphi_2(\xi(M)) = \varphi_2(\{u_i^{(1)},v_i^{(2)},v_i^{(3)}\}) = 
\{u_i^{(1)},\varphi(\{v_i^{(2)},v_i^{(3)}\})\} = 
\{u_i^{(1)},w_i^{(2)},w_i^{(3)}\},
\]
with $v_i^{(2)} = \sum_{c_r=i} 1$, $v_i^{(3)} = \sum_{d_r=i} 1$, and $w_j^{(l)} = v_i^{(l)}$ if $j=\varphi(i)$ for some $i$, and 0 otherwise.

If $j=\varphi(i)$ then $w_j^{(2)}=v_{\varphi^{-1}(j)}^{(2)}=\sum_{c_r=\phi^{-1}(j)}1 = \sum_{\varphi(c_r)=j}1 =u_j^{(2)}$. Moreover  $w_j^{(2)}=u_j^{(2)}=0$ when $j\notin \varphi(\mathbb{N})$. In the same manner we deal with $w_j^{(3)}$. Thus $\varphi_1(\xi(x))=\xi(\widetilde{\varphi}_1(x))$ and $\varphi_2(\xi(x))=\xi(\widetilde{\varphi}_2(x))$, and the lemma is proved.
\end{proof}

The orders defined above are not quite well behaved when multiplying by some monomial. Nevertheless there are several situations where they do behave, and these are just the ones we need.

\begin{lemma}
\label{mult}
Let $M=y_1^{r_1}\cdots y_n^{r_n}$, $N$, and $\widetilde{N}\in B$ be such that $N\leq \widetilde{N}$. Then $MN\leq M\widetilde{N}$.
\end{lemma}

\begin{proof}
If $N\in B_1$ and $\widetilde{N}\in B_2$ then $MN\in B_1$, $M\widetilde{N}\in B_2$, so we have the conclusion. The same holds if $N=\widetilde{N}$. Let $N=y_1^{u_1}\cdots y_k^{u_k}$ and $\widetilde{N}=y_1^{v_1}\cdots y_k^{v_k}$ where $u_kv_k\neq 0$. Then, for some $t_0\in \mathbb{N}$, we have $u_{t_0}<v_{t_0}$ and $u_t=v_t$ for every $t>t_0$. Denote $m=\max\{k,n\}$, then
\[
  \xi(MN)=\xi(y_1^{u_1+r_1}\cdots y_m^{u_m+r_m})=\{u_i+r_i\},
\]
\[
 \xi(M\widetilde{N})=\xi(y_1^{v_1+r_1}\cdots y_m^{v_m+r_m})=\{v_i+r_i\},
\]
where $u_{t_0}+r_{t_0}<v_{t_0}+r_{t_0}$ and $u_{t}+r_{t}=v_{t}+r_{t}$, for every $t>t_0$. It follows  $\xi(MN)<\xi(M\widetilde{N})$ if $N$, $\widetilde{N}\in B_1$. 

Now we consider $N$, $\widetilde{N}\in B_2$. Let $N=y_{a_1}\cdots y_{a_n}z_{c_1}z_{d_1}\cdots z_{c_k}\widehat{z_{d_k}}$ and $\widetilde{N}=y_{a'_1}\cdots y_{a'_m}z_{c'_1}z_{d'_1}\cdots z_{c'_s}\widehat{z_{d'_s}}$. Note that $z_{c_1}z_{d_1}\cdots z_{c_k}\widehat{z_{d_k}}\leq z_{c'_1}z_{d'_1}\cdots z_{c'_s}\widehat{z_{d'_s}}$. Then for every $M$, $M'\in B_1$, one has $Mz_{c_1}z_{d_1}\cdots z_{c_k}\widehat{z_{d_k}}\leq M'z_{c'_1}z_{d'_1}\cdots z_{c'_s}\widehat{z_{d'_s}}$. Therefore we suppose that
\[
z_{c_1}z_{d_1}\cdots z_{c_k}\widehat{z_{d_k}}= z_{c'_1}z_{d'_1}\cdots z_{c'_s}\widehat{z_{d'_s}}, \quad 
 y_{a_1}\cdots y_{a_n} \leq y_{a'_1}\cdots y_{a'_m}.
\]
But as seen above, $My_{a_1}\cdots y_{a_n} \leq M y_{a'_1}\cdots y_{a'_m}$, hence $MN\leq M\widetilde{N}$.
\end{proof}

\begin{lemma}
\label{mult1}
Let $M=y_1^{u_1}\cdots y_n^{u_n}$ and $N=y_1^{r_1}\cdots y_k^{r_k}$, then $M\leq MN$.
\end{lemma}
 \begin{proof}
Since 1, $N$, $M\in B_1$, Lemma~\ref{mult} yields $M \cdot 1\leq MN$. 
\end{proof}
We should note that, after a reordering of variables, each monomial $M \in D\langle X\rangle$ can be written in the form $M=\pm \delta(M)$ (modulo $I$), where $\delta(M)\in B$ .
\begin{lemma}
\label{mult2}
Let $M=z_{a_1}z_{a_2}\cdots z_{a_k}$ and $N=z_{t_1}\cdots z_{t_k}$. Then $M\leq MN$.
\end{lemma}
 \begin{proof}
Modulo the identities  from Lemma~\ref{basicids} we can transform $M$ and $N$ into their canonical form in $D\langle X\rangle/I$, without even changing the sign. Hence we can suppose $M\in B$ and $MN\in B$. We apply $\xi\colon   \partial  (\xi(MN))=\partial(\xi(M))+\partial(N)\geq \partial(\xi(M))$. One attains equality only when $N=1$, thus $M\le MN$.
 \end{proof}

\begin{lemma}
\label{mult3}
Let $M=y_1^{r_1}\cdots y_n^{r_n}$, $N=z_{t_1}\cdots z_{t_k}$, and $P$, $\widetilde{P}\in B$. If $P\leq \widetilde{P}$, then:
 \begin{itemize}
  \item [a)] $\delta(PM)\leq \delta(\widetilde{P}M)$;
  \item [b)] $PN\leq \widetilde{P}N$.
 \end{itemize}
\end{lemma}
\begin{proof}
a) If $P$, $\widetilde{P}\in B_1$, then $PM=MP$ and $\widetilde{P}M=M\widetilde{P}$. Then by Lemma~\ref{mult} we get $\delta(PM)\leq \delta(\widetilde{P}M)$. Now assume that $P = M'z_{c_1}z_{d_1}\cdots z_{c_m}\widehat{z_{d_m}}$ and $\widetilde{P} = M''z_{c'_1}z_{d'_1}\cdots z_{c'_l}\widehat{z_{d'_l}}\in B_2$, where $M'=y_{a_1}\cdots y_{a_{m_1}}$, $M''=y_{a'_1}\cdots y_{a'_{l_1}}$. Note that $\delta(PM)\leq \delta(\widetilde{P}M)$ in $D\langle X\rangle/I$, if $z_{c_1}z_{d_1}\cdots z_{c_m}\widehat{z_{d'_m}}< z_{c'_1}z_{d'_1}\cdots z_{c'_l}\widehat{z_{d'_l}}$. Therefore we assume that  $N=z_{c_1}z_{d_1}\cdots z_{c_m}\widehat{z_{d_m}} = z_{c'_1}z_{d'_1}\cdots z_{c'_l}\widehat{z_{d'_l}}$.  Then $M'\leq M''$ since $P\leq \widetilde{P}$. On the other hand, modulo $I$, we write $\delta(PM) = M'Mz_{c_1}z_{d_1}\cdots z_{c_m}\widehat{z_{d_m}}$ and $\delta(\widetilde{P}M)= M''Mz_{c_1}z_{d_1}\cdots z_{c_l}\widehat{z_{d_l}}$. It follows from Lemma~\ref{mult} that $M'M<M''M$. Thus we obtain $\delta(PM)\leq \delta(\widetilde{P}M)$.

Finally if $P\in B_1$ and $\widetilde{P}\in B_2$, the definition of the order $\leq$ yields that modulo $I$ one has $\delta(PM)\leq \delta(\widetilde{P}M)$, and this gives us the first statement.

b) Let us note that the product $PM$ is obtained by the juxtaposition of $P$ and $N$, and then reordering the variables $z_{t_i}$ (and only these variables), while keeping the same sign. If $P$, $\widetilde{P}\in B_1$, then  $PN\leq\widetilde{P}N$. Let $P=MN'$ and $\widetilde{P}=M'N''$, with $N'=z_{c_1}z_{d_1}\cdots z_{c_k}\widehat{z_{d_k}}$, $N''=z_{c'_1}z_{d'_1}\cdots z_{c'_k}\widehat{z_{d'_k}}$, and $M$, $M'$ monomials in the $y_i$'s. Therefore $PN=MN'N$ and $\widetilde{P}N=M'N''N$. If $N'=N''$ we obtain $M\leq M'$ hence $PN\leq \widetilde{P}N$. Suppose we have a strict inequality, then $N'\leq N''$. So it remains to show in this case that $N'N\leq N''N$, because $PN=MN'N \leq M'N''N=\widetilde{P}N$. Thus we can suppose $P=z_{c_1}z_{d_1}\cdots z_{c_k}\widehat{z_{d_k}}$ and $\widetilde{P}=z_{c'_1}z_{d'_1}\cdots z_{c'_{k'}}\widehat{z_{d'_{k'}}}$. But $P\leq \widetilde{P}$, that is, $\xi(P)\leq \xi(\widetilde{P})$. We obtain $\{(r_i,s_i)\}\leq \{(r_i',s_i')\}$, where $r_i=\sum_{c_j=i}1$, $s_i=\sum_{d_j=i}1$, $r_i'=\sum_{c_j'=i}1$, and $s_i'=\sum_{d_j'=i}1$. Now we are left with the following cases:
 \begin{itemize}
  \item [(a)] $\partial(\{(r_i,s_i)\})< \partial(\{(r_i',s_i)'\})$; 
  \item [(b)] $\partial(\{(r_i,s_i)\}) = \partial(\{(r_i',s_i)'\})$ and $\{s_i\}<\{s_i'\}$, or $\{s_i\} = \{s_i'\}$ and $\{r_i\}<\{r_i'\}$.
 \end{itemize}
In the first case, $PN\leq \widetilde{P}N$, and we can assume the equality $k=k'$ holds. In the second case, we have to consider either $P=z_{c_1}z_{d_1}\cdots z_{c_k}\widehat{z_{d_k}}$ and $\widetilde{P}=z_{c'_1}z_{d'_1}\cdots z_{c'_k}\widehat{z_{d'_k}}$ both of even degree, or both of odd degree. Let us suppose that $P$ and $\widetilde{P}$ are of even degree. If we show that the statement holds for $N=z_{t_1}$, then the general case follows by an obvious induction.

We have $\{s_i\}<\{s_i'\}$, or $\{s_i\} = \{s_i'\}$ and $\{r_i\}<\{r_i'\}$, where $r_i=\sum_{c_j=i}1$, $s_i=\sum_{d_j=i}1$, $r_i'=\sum_{c_j'=i}1$, and $s_i'=\sum_{d_j'=i}1$. Then $\xi(PN)=\{(u_i,v_i)\}$ with $u_i=r_i$ in case $i\neq t_1$, $u_{t_1}=r_{t_1}+1$ and $v_i=s_i$, $i\geq 1$. In the same manner for $\widetilde{P}$ we obtain $\xi(\widetilde{P}N)=\{(u'_i,v'_i)\}$, where $u'_{t_1}=r'_{t_1}+1$, $u'_i=r'_i$ if $i\neq t_1$ and $v'_i=s'_i$, $i\geq 1$. If $\{s_i\}<\{s_i'\}$, then $\{v_i\}<\{v'_i\}$; if $\{s_i\} = \{s_i'\}$ and $\{r_i\}<\{r_i'\}$, then $\{v_i\}=\{v'_i\}$ and $\{u_i\}<\{u'_i\}$. Therefore $\{(u_i,v_i)\}\leq \{(u'_i,v'_i)\}$, that is $\xi(PN)\leq \xi(\widetilde{P}N)$. Thus  $PN\leq\widetilde{P}N$. The case when $P$ and $\widetilde{P}$ are of odd degree is analogous.
\end{proof}

In the following lemma we prove that $\widetilde{\varphi}_1$, $\widetilde{\varphi}_2$ and $\widetilde{\varphi}$ preserve $\leq$. Recall that these functions were defined in the proof of Lemma~\ref{Vispwos}.

\begin{lemma}
\label{mult4}
 Let $\widetilde{\varphi}\in \widetilde{\Phi}$, $\widetilde{\varphi}_i\in \widetilde{\Phi}_i$, $i=1$, $2$. If $M<\widetilde{M}$, then:
 \begin{itemize}
  \item [a)] $\widetilde{\varphi}_i(M)<\widetilde{\varphi}_i(\widetilde{M})$;
  \item [b)] $\widetilde{\varphi}(M)<\widetilde{\varphi}(\widetilde{M})$.
 \end{itemize}
\end{lemma}

\begin{proof}
We have $\widetilde{\varphi}=\widetilde{\varphi}_1\circ \widetilde{\varphi}_2$, hence it suffices to prove (a). Let first $M$, $\widetilde{M}\in B_1$. If $M=1$ or $\widetilde{M}=1$, the statement is obvious. Hence we suppose $M=y_1^{r_1}\cdots y_n^{r_n}$ and $\widetilde{M}=y_1^{s_1}\cdots y_n^{s_n}$, $s_n\neq 0$ or $r_n\neq 0$. But $\xi(\widetilde{\varphi}_1(M))=\xi(y_{\varphi(1)}^{r_1}\cdots y_{\varphi(n)}^{r_n})=\{r'_i\}$, with $r'_i=r_j$ if $i=\phi(j)$, and $r'_i=0$ if $i\notin \varphi(\mathbb{N})$. In the same way,  $\xi (\widetilde{\varphi}_1(\widetilde{M}))=\xi(y_{\phi(1)}^{s_1}\cdots y_{\phi(n)}^{s_n})=\{s'_i\}$, where $s'_i=s_j$ if $i=\phi(j)$, and $s'_i=0$ if $i\notin \varphi(\mathbb{N})$.

But $M<\widetilde{M}$, thus there exists $t_0\in \mathbb{N}$ with $r_{t_0}<s_{t_0}$ and $s_t=r_t$ when $t>t_0$. It is enough to prove that $r'_{\varphi(t_0)}<s'_{\varphi(t_0)}$, and $s'_t=r'_t$ if $t>\phi(t_0)$. Since $\varphi({t_0})\in \varphi({\mathbb{N}})$ we obtain $r'_{\varphi(t_0)}=r_{t_0}<s_{t_0}=s'_{\varphi(t_0)}$. Choose $t>\varphi(t_0)$. If $t\in \varphi(\mathbb{N})$, since $\varphi$ is increasing, there is some $p>t_0$ such that $t=\varphi(p)$. Therefore $r'_t=r_{p}=s_{p}=s'_t$. Also, if $t\notin \varphi(\mathbb{N})$ then $r'_t=s'_t=0$. Hence $\widetilde{\varphi_1}(M)<\widetilde{\varphi_1}(\widetilde{M})$. 

Let $M=NP$ e $\widetilde{M}=N'P'\in B$, where $N$, $N'\in B_1$, and $P$, $P'$ are monomials in the variables $z_{t}$. We know that $\widetilde{\varphi_1}$ fixes the $z_t$, thus $\widetilde{\varphi_1}(M)=\widetilde{\varphi_1}(N)P$, and $\widetilde{\varphi_1}(\widetilde{M})=\widetilde{\varphi_1}(N')P'$. If $P<P'$ we are done, so suppose $P=P'$. But $M<\widetilde{M}$, then $N<N'$ and also $\widetilde{\varphi_1}(N)<\widetilde{\varphi_1}(N')$ by the above argument. Therefore $\widetilde{\varphi_1}(M)=\widetilde{\varphi_1}(N)P< \widetilde{\varphi_1}(N')P'=\widetilde{\varphi_1}(\widetilde{M})$.

Now we consider the function $\widetilde{\varphi_2}$. Let first $M=z_{c_1}z_{d_1}\cdots z_{c_m}\widehat{z_{d_m}}$ and $\widetilde{M}=z_{c'_1}z_{d'_1}\cdots z_{c'_{l}}\widehat{z_{d'_{l}}}$. Then we have $\xi(M)=\{(0,r_i,s_i)\}$ and $\xi(\widetilde{M})=\{(0,r'_i,s'_i)\}$, where $r_i=\sum_{c_j=i}1$, $s_i=\sum_{d_j=i}1$, $r'_i=\sum_{c'_j=1}1$, $s'_i=\sum_{d'_j=1}1$. Furthermore $\xi(\widetilde{\varphi_2}(M))=\{(0,u_i,v_i)\}$ and $\xi(\widetilde{\varphi_2}(\widetilde{M}))=\{(0,u'_i,v'_i)\}$, with $(u_i,v_i)=(r_j,s_j)$, and $(u'_i,v'_i)=(r'_j,s'_j)$ if $i=\varphi(j)$, and $(u_i,v_i)=(u'_i,v'_i)=(0,0)$ if $i\notin\varphi(\mathbb{N})$. Since $M$, $\widetilde{M}\in B_2$, and by assumption  $M<\widetilde{M}$, it follows there are two possibilities only:
 \begin{itemize}
  \item [(a)] $\partial(\{(r_j,s_j)\})<\partial(\{(r'_j,s'_j)\})$; 
  \item [(b)] $\partial(\{(r_j,s_j)\})=\partial(\{(r'_j,s'_j)\})$, and either $\{s_i\}<\{s'_i\}$, or $\{s_i\}=\{s'_i\}$ and $\{r_i\}<\{r'_i\}$.
 \end{itemize}
But $\partial(\{(0,u_i,v_i)\})=\partial(\{(0,r_j,s_j)\})$ and $\partial(\{(0,u'_i,v'_i)\})=\partial(\{(0,r'_j,s'_j)\})$; if (a) takes place then $\partial(\{(0,u_i,v_i)\})<\partial(\{(0,u'_i,v'_i)\})$. If (b) takes place then $\{v_i\}<\{v'_i\}$, or $\{v_i\}=\{v'_i\}$ and $\{u_i\}<\{u'_i\}$. Therefore $\xi(\widetilde{\varphi}_2(M))<\xi(\widetilde{\varphi}_2(\widetilde{M}))$.

Finally let $M=NP$ and $\widetilde{M}=N'P'\in B$, with $N$, $N'\in B_1$ and $P$,  $P'$ monomials in the variables $z_{t}$. If $P<P'$, then $\widetilde{\varphi_2}(P)<\widetilde{\varphi_2}(P')$. Since $\widetilde{\varphi_2}$ fixes the variables  $y$, it follows $\widetilde{\varphi_2}(M)=N\widetilde{ \varphi_2}(P) <N'\widetilde{\varphi_2}(P')=\widetilde{\varphi_2}(N'P')=\widetilde{\varphi_2}(\widetilde{M})$. If $P=P'$ then $N<N'$, thus $\widetilde{\varphi_2}(M)= N\widetilde{\varphi_2}(P)<N'\widetilde{\varphi_2}(P')=\widetilde{\varphi_2}(N'P')=\widetilde{\varphi_2}(\widetilde{M})$.
\end{proof}

\begin{lemma}
\label{mult5}
Suppose $M$, $\widetilde{M}\in B_i$, $i=1$, 2. If $M\leq'\widetilde{M}$ then there exist $\widetilde{\varphi}\in\widetilde{\Phi}$, $N\in B_1$, and $P=z_{t_1}\cdots z_{t_k}$, $k\geq 0$ with $N\widetilde{\varphi}(M)P=\widetilde{M}$.
\end{lemma}
\begin{proof}
First we consider $M$, $\widetilde{M}\in B_1$, $M=y_1^{r_1}\cdots y_n^{r_n}$ and $\widetilde{M}=y_1^{s_1}\cdots y_n^{s_n}$ where $r_n\neq 0$ or $s_n\neq 0$. As $M\leq'\widetilde{M}$, there is $\varphi\in \Phi$ such that $r_j\leq s_{\phi(j)}$, for every $j\geq 1$. We also have $\xi(\widetilde{\varphi}(M))=\{r'_i\}$ where $r'_i=r_j$ if $i=\varphi(j)$ for some $j$, and $r'_i=0$ otherwise. Therefore $\widetilde{\varphi}(M)=y_1^{r'_1}\cdots y_{\varphi(n)}^{r'_{\varphi(n)}}$ since $\varphi$ is increasing. For every $i\geq 1$ we denote $l_i=s_i-r'_i$. Then $l_i=s_i$ if $i\notin \varphi(\mathbb{N})$ and $l_i=s_{\phi(j)}-r_j>0$ if $i=\varphi(j)$. Hence we can consider $N=y_1^{l_1}\cdots y_{\varphi(n)}^{l_{\varphi(n)}}$. For $\xi(N\widetilde{\varphi}(M))$ we obtain
 \begin{align*} 
\xi(N\widetilde{\varphi}(M)) & = \xi(y_1^{l_1}\cdots y_{\phi(n)}^{l_{\phi(n)}}\widetilde{\varphi}(M)) = \xi(y_1^{l_1}\cdots y_{\varphi(n)}^{l_{\varphi(n)}}y_1^{r'_1}\cdots y_{\phi(n)}^{r'_{\phi(n)}})\\
 & = \xi(y_1^{l_1+r'_1}\cdots y_{\varphi(n)}^{l_{\varphi(n)}+r'_{\varphi(n)}}) = \{l_i+r'_i\}=\{s_i\}=\xi(\widetilde{M}). 
\end{align*}
Therefore $N\widetilde{\varphi}(M)P=\widetilde{M}$, where $P=1$. Let now $M=\overline{N}\,\overline{P}$, $\widetilde{M}=\widetilde{N}\,\widetilde{P}\in B_2$, where $\overline{N}=y_{a_1}\cdots y_{a_k}$, $\widetilde{N}=y_{a'_1}\cdots y_{a'_{k'}}$, $\overline{z}=z_{c_1}z_{d_1}\cdots z_{c_m}\widehat{z_{d_m}}$, and $\widetilde{P}=z_{c'_1}z_{d'_1}\cdots z_{c'_{m'}}\widehat{z_{d'_{m'}}}$. 
We have to show there exist $\widetilde{\varphi}\in \widetilde{\Phi}$ and $P=z_{t_1}\cdots z_{t_i}$ satisfying the following conditions:
\begin{itemize}
 \item [(a)] $N\widetilde{\varphi}(\overline{N})=\widetilde{N}$; 
 \item [(b)] $\widetilde{\varphi}(\overline{P})P=\widetilde{P}$.
\end{itemize}
Observe that these conditions imply $N\widetilde{\varphi}(M)P= N \widetilde{\varphi}(\overline{N})\widetilde{\varphi}(\overline{P})P = \widetilde{N}\,\widetilde{P}=\widetilde{M}$. 

We recall that $\xi(M)=\{(u^{(1)}_i, u^{(2)}_i, u^{(3)}_i)\}$ where $u^{(1)}_i=\sum_{a_j=i}1$, $u^{(2)}_i=\sum_{c_j=i}1$, $u^{(3)}_i= \sum_{d_j=i}1$. Analogously $\xi(\widetilde{M})=\{(v^{(1)}_i, v^{(2)}_i, v^{(3)}_i)\}$ where $v^{(1)}_i=\sum_{a'_j=i}1$, $v^{(2)}_i=\sum_{c'_j=i}1$, $v^{(3)}_i = \sum_{d'_j=i}1$. 

(a) Since $M\leq' \widetilde{M}$, there is some $\varphi\in \Phi$ such that $u^{(i)}_j<v^{(i)}_{\varphi(j)}$, $j\geq 1$, and $i\in \{1,2,3\}$. Suppose  $l_k^{(i)}=v^{(i)}_k$ if $k\notin \varphi(\mathbb{N})$, and $l^{(i)}_k=v^{(i)}_{\varphi(j)}-u^{(i)}_j$ if $k=\varphi(j)$, for every $j$, $k\geq 1$, and $i\in \{1,2,3\}$. We observe $l^{(i)}_k\geq 0$, thus we consider a monomial $N$ such that $\xi(N)=\{l^{(1)}_j\}$. We also have $\overline{N}\leq'\widetilde{N}$, since $u^{(1)}_j <v^{(1)}_{\varphi(j)}$, $j\geq 1$. The first part of the proof yields $N\widetilde{\varphi}(\overline{N})=\widetilde{N}$. As for (b), take the element  $\overline{P}=z_{c_1}z_{d_1}\cdots z_{c_m}z_{d_m}$, and set $P=z_{e_1}z_{f_1}\cdots z_{e_{m''}}z_{f_{m''}}$, where $l^{(2)}_i=\sum_{e_j=i}1$ and $l^{(3)}_i=\sum_{f_j=i}1$. Therefore 
\[
\xi(\widetilde{\varphi}(\overline{P})P) = \xi(z_{\varphi{(c_1)}}z_{\varphi{(d_1)}}\cdots z_{\varphi{(c_m)}}z_{\varphi{(d_m)}}z_{e_1}z_{f_1}\cdots z_{e_{m''}}z_{f_{m''}}) = \{(0,u'^{(2)}_i, u'^{(3)}_i)\}.
\]
Here $u'^{(2)}_i=\sum_{\varphi(c_j)=i}1+ \sum_{e_j=i}1$, and $u'^{(3)}_i= \sum_{\varphi(d_j)=i}1+ \sum_{f_j=i}1$. If $i=\varphi(s)$ for some $s$ then   
$u'^{(2)}_i = \sum_{\varphi(c_j)=\varphi(s)}1 + \sum_{e_j=\varphi(s)}1 = u^{(2)}_s+l^{(2)}_{\phi(s)}=v^{(2)}_i$. Also if $i\notin \varphi(\mathbb{N})$ then $u'^{(2)}_i=0+\sum_{e_j=i}1 =l^{(2)}_i=v^{(2)}_i$. Analogously $u'^{(3)}_i=v^{(3)}_i$, $i\geq 1$. Then $\xi(\widetilde{P})=\{(0,v^{(2)}_i, v^{(3)}_i)\}=\{(0,u'^{(2)}_i, u'^{(3)}_i)\}=\xi(\widetilde{\varphi}(\overline{P})P)$. This means $\widetilde{\varphi}(\overline{P})P = \widetilde{P}$. If, finally, $\overline{P}=z_{c_1}z_{d_1}\cdots z_{c_m}$ we put $P=z_{f_1}z_{e_1}\cdots z_{f_{m''}}$ and repeat the above argument.  
\end{proof}

\begin{definition}
 Let $f=\sum_{i=1}^n\alpha_im_i\in D\langle X\rangle/I$, $m_i\in B$, and $\alpha_i\in D$. If $\alpha_1\neq 0$ and $m_1>m_i$ for every $i\geq 1$, then the leading monomial, coefficient, and term of $f$ are, respectively, $lm(f)=m_1$, $lc(f)=\alpha_1$, and $lt(f)=\alpha_1m_1$.
\end{definition}

\begin{lemma}
\label{mult6}
 Let $f$ and $\widetilde{f}$ be linear combinations of elements of $B$, and let  $lm(f)=M$ e $lm(\widetilde{f})=\widetilde{M}$. If $M\leq' \widetilde{M}$, then there exist $\widetilde{\varphi}\in \widetilde{\Phi}$, $N\in B_1$, and $P\in B$ such that $lm(N\widetilde{\varphi}(f)P)=\widetilde{M}$.
\end{lemma}

\begin{proof}
As $M\leq \widetilde{M}$, by Lemma~\ref{mult4}, there are $\widetilde{\varphi}\in\widetilde{\Phi}$, $N\in B_1$, and $P=z_{t_1}\cdots z_{t_k}$, $k\geq 0$ such that $N\widetilde{\varphi}(M)P = \widetilde{M}$. We saw that the order $\leq$ is preserved under $\widetilde{\varphi}$, and also by multiplications from the left by $N$, and from the right, by $P$. Therefore $lm(N\widetilde{\varphi}(f)P) =  N \widetilde{\varphi}(M)P = \widetilde{M}$.
\end{proof}

The following is our main result in the paper.

\begin{theorem}
Let $D$ be an associative, commutative, unital Noetherian domain. Let $I$ be the   graded ideal of weak identities generated by the identities $[y_1,y_2]$, $z_1z_2z_3-z_3z_2z_1$, and $yz+zy$. Then every graded ideal of weak identities containing $I$, is finitely generated as a graded ideal of weak identities. In other words the graded identities for the pair $(M_2(D), sl_2(D))$ satisfy the Specht property.
\end{theorem}

 \begin{proof}
We suppose that, on the contrary, there exists a graded ideal $J$ of weak identities that contains $I$, and that $J$ is not finitely generated. 

Thus there exists $g_1\in J\setminus I$ such that for the weak graded ideal $I_1$ generated by $g_1$ and $I$ we have $ I \subsetneq I_1\subsetneq J$. But $J$ is not finitely generated, hence we choose some $g_2\in J\setminus I_1$, and then generate a weak graded ideal with $I_1$ and $g_2$, and so on. Thus we obtain an ascending chain $ I \subsetneq I_1\subsetneq I_2\subsetneq \cdots $ of weak graded ideals. It induces another ascending chain of weak graded ideals $ J_1\subsetneq J_2\subsetneq \cdots $ in $D\langle X \rangle/I$ where $J_k=I_k/I$. Below we show that this is impossible.

Since $B$ is a basis of the $D$-module $D\langle X\rangle/I$, we write every element of $R_i=J_i\setminus J_{i-1}$ as a linear combination of elements of $B$. Denote by $L_i$ the set of the leading monomials of the elements of $R_i$. Since $R_i\neq \emptyset$ we have  $L_i\neq \emptyset$. But $(B,\leq)$ is linearly well ordered, hence each $L_i$ has unique least element $M_i$. Moreover $(B,\leq')$ is partially well ordered, thus the sequence $\{M_i\}$ has an infinite subsequence $\{M_{i_k}\}$ which is non-decreasing with respect to $\leq'$, by Higman's theorem \ref{higman}.

For each $M_{i_l}\in L_{i_l}$, we choose a non-zero element $h_l\in R_{i_l}$ whose leading term is $lt(h_l)=\alpha_{l}M_{i_l}$, $\alpha_l\in D$. Let $I_0$ be the ideal generated by $\alpha_1$, $\alpha_2$, \dots, in $D$. As $D$ is  Noetherian, there exists $k$ such that $I_0$ is generated by $\alpha_1$, \dots, $\alpha_k$. It follows that $\alpha_{k+1}=\sum_l^k\beta_l\alpha_l$, for some $\beta_l\in D$. Since $M_{i_l}\leq' M_{i_{k+1}}$, for each $l=1$, \dots, $k$, there exist $N_l$, $P_l\in B$ and $\widetilde{\varphi}\in \widetilde{\Phi}$ such that $lm(N_l\widetilde{\varphi}(h_l)P_l)=M_{i_{k+1}}$, by Lemma~\ref{mult6}. Therefore  $lt(N_l\widetilde{\varphi}(h_l)P_l)=\alpha_{l}M_{i_{k+1}}$. We consider $h=\sum_{l=1}^k\beta_lN_l\widetilde{\varphi}(h_l)P_l$, and note that $h=\sum_{l=1}^k\beta_l\alpha_lM_{i_{k+1}}+$ (summands that are less than $M_{i_{k+1}}$), hence $lt(h)=\sum_{l=1}^k\beta_l\alpha_lM_{i_{k+1}}=(\sum_{l=1}^k\beta_l\alpha_l)M_{i_{k+1}}=\alpha_{k+1}M_{i_{k+1}}$.

For each $l=1$, \dots, $k$, we have $h_l\in R_{i_l}\subsetneq J_{i_l}\subseteq J_{i_{(k+1)}-1}$, since $i_1<i_2<\cdots <i_k\leq i_{(k+1)}-1$. Also $\widetilde{\varphi}(h_l)\in J_{i_{(k+1)}-1}$, since $I_k$ is an ideal of weak graded identities, and as such is closed under the endomorphisms from $\widetilde{\Phi}$. Thus $N_l\widetilde{\varphi}(h_l)P_l\in J_{i_{(k+1)}-1}$ and  $h\in J_{i_{(k+1)}-1}$. 

Now $f=h_{k+1}-h\in D\langle X\rangle/I$, and the leading term of $h_{k+1}$ is $\alpha_{k+1}M_{i_{k+1}}$. This is the same leading term as that of $h$, hence the leading monomial of $f$ is less than that of $M_{i_{k+1}}$. By the choice of the  $h_l$, we have $h_{k+1}\in J_{i_{(k+1)}}\setminus J_{i_{(k+1)}-1}$. Since $h\in J_{i_{(k+1)}-1}\subsetneq J_{i_{(k+1)}}$, it follows $f\in J_{i_{(k+1)}}\setminus J_{i_{(k+1)}-1}= R_{i_{k+1}}$. But this is impossible since the leading monomial of $f$ lies in $L_{i_{k+1}}$ by definition, and it is less than  $M_{i_{k+1}}$,  the least element in $L_{i_{k+1}}$. 
 \end{proof}

\end{document}